\documentclass[12pt,fleqn]{article}
\usepackage{latexsym}
\usepackage{amsmath}
\usepackage{amssymb}
\usepackage{epsfig}
\usepackage{amsfonts}

\begin{document}
\newtheorem{definition}{Definition}[section]
\newtheorem{theorem}[definition]{Theorem}
\newtheorem{lemma}[definition]{Lemma}
\newtheorem{proposition}[definition]{Proposition}
\newtheorem{examples}[definition]{Examples}
\newtheorem{corollary}[definition]{Corollary}
\def\square{\Box}
\newtheorem{remark}[definition]{Remark}
\newtheorem{remarks}[definition]{Remarks}
\newtheorem{exercise}[definition]{Exercise}
\newtheorem{example}[definition]{Example}
\newtheorem{observation}[definition]{Observation}
\newtheorem{observations}[definition]{Observations}
\newtheorem{algorithm}[definition]{Algorithm}
\newtheorem{criterion}[definition]{Criterion}
\newtheorem{algcrit}[definition]{Algorithm and criterion}

\newenvironment{prf}[1]{\trivlist
\item[\hskip \labelsep{\it
#1.\hspace*{.3em}}]}{~\hspace{\fill}~$\square$\endtrivlist}
\newenvironment{proof}{\begin{prf}{Proof}}{\end{prf}}

\title{The unicity of real Picard--Vessiot fields}
\author{Teresa Crespo, Zbigniew Hajto and Marius van der Put }
\date{}
\maketitle

\begin{abstract} Using Deligne's work on Tannakian categories, the unicity
of real Picard-Vessiot fields for differential modules over a real differential field is derived. The inverse problem for real forms of a semi-simple
group is treated. Some examples illustrate the relations between differential modules, Picard--Vessiot fields and real forms of a group.
\let\thefootnote\relax\footnotetext{
T. Crespo and Z. Hajto acknowledge support of grant MTM2009-07024, Spanish Science Ministry.}     \end{abstract}

\section{Introduction}

$K$ denotes a real differential field with field of constants $k$. We suppose that $k\neq K$ and
that $k$ is a real closed field. Let $M$ denote a differential module over $K$ of dimension $d$,
represented by a matrix differential equation $y'=Ay$ where $A$ is a $d\times d$-matrix with entries
in $K$. A {\it Picard--Vessiot field} $L$ for $M/K$ is a field extension of $K$ such that:\\
(a) $L$ is equipped with a differentiation extending the one of $K$,\\
(b)  $M$ has a full space of solutions over $L$, i.e., there exists an invertible $d\times d$-matrix $F$ (called
a fundamental matrix) with entries in $L$ satisfying $F'=AF$,\\
(c) $L$ is (as a field) generated over $K$ by the entries of $F$,\\
(d) the field of constants of $L$ is again $k$.\\

A {\it real Picard--Vessiot field} $L$ for $M/K$ is a
Picard--Vessiot field which is also a real field. In \cite{CHS1}
and \cite{CHS2} the existence of a real Picard--Vessiot field is
proved using results of Kolchin.

The main result of this paper is:
\begin{theorem} Let $L_1,L_2$ denote two real Picard--Vessiot extensions for $M/K$. Suppose that $L_1$ and $L_2$
have total orderings which induce the same total ordering on $K$. Then there exists a $K$-linear isomorphism
$\phi : L_1\rightarrow L_2$ of differential fields.
\end{theorem}

\begin{remarks} {\rm
Suppose that $\phi$ exists.  Choose a total ordering of $L_1$ and define the total ordering of $L_2$ to be induced
by $\phi$. Then $L_1$ and $L_2$ induce the same total ordering on $K$. Therefore the condition of the Theorem 1.1
is necessary.

If $K$ happens to be real closed, then the assumption in the theorem is superfluous since $K$ has a unique total ordering. On the other hand, consider the example $K=k(z)$ with differentiation $'=\frac{d}{dz}$ and the equation
$y'=\frac{1}{2z}y$. Let $L_1=K(t_1)$ with $t_1^2=z$ and $L_2=K(t_2)$ with $t_2^2=-z$. Both fields are real
Picard--Vessiot fields for this equation. They are not isomorphic as differential field extensions of $K$, since
$z$ is positive for any total ordering of $L_1$ and $z$ is negative for any total ordering of $L_2$.}
\hfill $\square$ \\ \end{remarks}

The proof of Theorem 1.1 uses Tannakian categories as presented in
\cite{D-M} and  P.~Deligne's fundamental paper \cite{De}. We adopt
much of the notation of \cite{De}. Let $<M>_\otimes $ denote the
Tannakian category generated by the differential module $M$. The
forgetful functor $\rho : <M>_\otimes \rightarrow vect(K)$
associates to any differential module $N\in <M>_\otimes $ the
finite dimensional $K$-vector space $N$. Let $\omega : <M>_\otimes
\rightarrow vect(k)$
be a fibre functor with values in the category $vect(k)$ of the finite dimensional vector spaces over $k$.\\

 {\it Now we recall some results of}  \cite{De}, \S 9. The functor $\underline{Aut}^\otimes(\omega)$ is represented by a linear
 algebraic group $G$ over $k$.  By Proposition 9.3, the functor $\underline{Isom}^\otimes_K(K\otimes \omega ,\rho)$ is represented by a torsor $P$ over $G_K:=K\times _kG$.   This torsor is affine, irreducible and its coordinate ring $O(P)$ has a natural differentiation extending the differentiation of $K$. Moreover, the field of fractions $K(P)$ of $O(P)$ is
 a Picard--Vessiot field for $M/K$.

 On the other hand, let $L$ be a Picard--Vessiot field for $M/K$. Define the fibre functor
 $\omega _L: <M>_\otimes \rightarrow vect(k)$ by $\omega _L(N)=\ker (\partial : L\otimes _KN\rightarrow L\otimes_K N)$.
 Then $\omega _L$ produces a Picard--Vessiot field $L'$ which is isomorphic to $L$ as differential field extension
  of $K$. The conclusion is:\\

 \begin{proposition}[{\rm \cite{De},\S 9}]
 The above constructions yield a bijection between the (isomorphy classes of) fibre functors
 $\omega : <M>_\otimes \rightarrow vect(k)$ and the (isomorphy classes of) Picard--Vessiot fields $L$ for $M/K$.
  \end{proposition}

The following result will also be useful.

\begin{proposition}[{\rm \cite{D-M}, Thm. 3.2}]  Let $\omega :<M>_\otimes \rightarrow vect(k)$ be a fibre functor and
$G=\underline{Aut}^\otimes_k(\omega)$.\\
\noindent {\rm (a)}. For any field $F\supset k$ and any fibre functor $\eta :  <M>_\otimes \rightarrow vect(F)$,
 the functor $\underline{Isom}^\otimes _F(F\otimes \omega ,\eta)$ is representable by a torsor over
$G_F=F\times _kG$.\\
{\rm (b)}. The map $\eta\mapsto  \underline{Isom}^\otimes _F(F\otimes \omega ,\eta)$ is a bijection between
the (isomorphy classes of) fibre functors  $\eta :  <M>_\otimes \rightarrow vect(F)$ and the (isomorphy classes of)
$G_F$-torsors.

 \end{proposition}

  The main ingredient in the proof of Theorem 1.1, given in \S 2, is:

\begin{theorem} Suppose that $K$ is real closed. Let $L$ be a Picard--Vessiot field for $M/K$. Then $L$ is a real
field if and only if the torsor $\underline{Isom}_K^\otimes (K\otimes \omega _L,\rho)$ is trivial.
\end{theorem}
\section{The proof of Theorem 1.1}
\subsection{Reduction to $K$ is a real closed differential field}
For notational convenience, the differential module $M/K$ is represented by a scalar homogeneous
linear differential equation $\mathcal{L}(y):=y^{(d)}+a_{d-1}y^{(d-1)}+\cdots +a_1y^{(1)}+a_0y=0$.
A Picard--Vessiot field $L$ for $\mathcal{L}$ has the properties:\\
$k$ is the field of constants of $L$, the {\it solution space}
$V=\{v\in L|\ \mathcal{L}(v)=0\}$ is a $k$-linear space of dimension $d$ and $L$ is generated over the field $K$ by $V$ and all the derivatives of the elements in $V$. One writes $L=K<V>$ for this last property.

\begin{lemma} Let $L_1,L_2$ be two real Picard--Vessiot fields for $M$ over $K$. Suppose that
 $L_1$ and $L_2$ have total orderings extending a total ordering $\tau$ on $K$. Let $K^r\supset K$ be the real closure of $K$ inducing the total ordering $\tau$. Then:\\

 The fields $L_1,L_2$ induce Picard--Vessiot fields
 $\tilde{L}_1,\tilde{L}_2$ for $K^r\otimes M$ over $K^r$. These fields
 are isomorphic as differential field extensions of $K^r$ if and only if
 $L_1$ and $L_2$ are isomorphic as differential field extensions of $K$.
\end{lemma}
\begin{proof} Let, for $j=1,2$, $\tau_j$ be a total ordering on $L_j$ inducing $\tau$ on $K$ and
let $L_j^r$ be the real closure of $L_j$ which induces the ordering $\tau_j$. The algebraic closure
$K_j$ of $K$ in $L_j^r$ is real closed. Since $\tau_j$ induces $\tau$, there exists a $K$-linear
isomorphism $\phi _j:K^r\rightarrow K_j$. This isomorphism is unique since the only $K$-linear
automorphism of $K^r$ is the identity. We will identify $K_j$ with $K^r$.

 Let $V_j\subset L_j$ denote the solution space of $M$. Then, for $j=1,2$, the field
 $\tilde{L}_j:=K^r<V_j>\subset L_j^r$ is a real Picard--Vessiot field for $K^r\otimes M$.

   Assume the existence of
 a $K^r$-linear differential isomorphism\linebreak $\psi :K^r<V_1>\rightarrow K^r<V_2>$.
 Clearly $\psi (V_1)=V_2$ and $\psi$ induces therefore a $K$-linear differential isomorphism
 $L_1=K<V_1>\rightarrow L_2=K<V_2>$.

 On the other hand, an isomorphism $\phi : L_1\rightarrow L_2$ (of differential field extensions of $K$) extends
 to an isomorphism $\tilde{\phi}: L_1^r\rightarrow L_2^r$. Clearly $\tilde{\phi}$ maps $\tilde{L}_1$ to $\tilde{L}_2$.
 \end{proof}

\subsection{Real algebras and connected linear groups}

 An algebra $R$ (commutative with 1 and without zero divisors) is called {\it real} if $x_1,\dots ,x_n\in R$ and
 $\sum _{j=1}^nx_j^2=0$ implies $x_1=\cdots =x_n=0$. By lack of a reference we give a proof of the following
 statement.

\begin{lemma} Let $F$ be a real closed field and let $G$ be a linear algebraic group over $F$ such that
$G_{F(i)}=F(i)\times _FG$ is connected. Then the coordinate ring $F[G]$ of $G$ over $F$ is a real algebra.
\end{lemma}
\begin{proof} We consider the case $F$ is equal to $\mathbb{R}$, the field of real numbers. Consider
$x_1,\dots ,x_n\in \mathbb{R}[G]$ with $\sum _{j=1}^nx_j^2=0$.  We regard $G(\mathbb{R})$ as a real analytic group. There is an exponential map $Lie(G)(\mathbb{R})\rightarrow G(\mathbb{R})$, where $Lie(G)$ is the Lie algebra of $G$.
Define the real analytic map $\tilde{x_j}:Lie(G)(\mathbb{R})\stackrel{exp}{\rightarrow} G(\mathbb{R})\stackrel{x_j}{\rightarrow} \mathbb{R}$. Now $\sum \tilde{x}_j^2=0$ and hence all $\tilde{x_j}=0$.
The complex analytic  morphism $X_j:Lie(G)(\mathbb{C})\stackrel{exp}{\rightarrow} G(\mathbb{C})\stackrel{x_j}{\rightarrow} \mathbb{C}$ is the complex extension of $\tilde{x_j}$. It is zero since it is zero on the subset
$Lie(G)(\mathbb{R})$ of $Lie(G)(\mathbb{C})$. The image of the complex exponential map generates the component of the identity of $G(\mathbb{C})$ and $x_j$ is zero on this set. By assumption $G_\mathbb{C}$ is connected and thus $x_j=0$ is zero for all $j$.  Hence $\mathbb{R}[G]$ is a real algebra. \\

For any real field $F$ which has an embedding in
$\mathbb{R}$, one has $F[G]\subset \mathbb{R}[G]$ and $F[G]$ is a
real algebra. Further, a real field $k$ which is finitely generated over $\mathbb{Q}$ has an embedding in $\mathbb{R}$ (\cite{Si} Proposition 3).

We consider the general case: $G$ is a linear algebraic group defined over a real closed field $F$. Now $G$ is defined over a subfield $F_0$
of $F$ which is finitely generated over $\mathbb{Q}$. Then $F[G]$ is the union of the subrings $k[G]$, where $k$ runs in the set of the subfields
of $F$ which are finitely generated over $\mathbb{Q}$ and contain $F_0$.
Thus $F[G]$ is a real algebra since every $k[G]$ is a real algebra.
\end{proof}

\noindent {\it Remark}.  The condition that $G_{F(i)}$ is connected (or equivalently $G$ is connected)
 is necessary. Indeed, consider the example of the group $\mu _3$ over $\mathbb{R}$ with coordinate algebra  $\mathbb{R}[X]/(X^3-1)$. This algebra is isomorphic to the direct sum
  $\mathbb{R}\oplus \mathbb{R}[X]/(X^2+X+1)$ and therefore is not real.

 \begin{corollary}[{\rm \cite{La} Corollary 6.8}] Let $G$ be a linear algebraic group over the real closed field $F$. Suppose that
 $G_{F(i)}$ is connected. Then the group $G(F)$ is Zariski dense in $G(F(i))$.
  \end{corollary}

 \begin{theorem}[{\bf 1.5}]  Suppose that $K$ is real closed. Let $L$ be a Picard--Vessiot field for a differential module
 $M/K$.  Then $L$ is a real field if and only if the torsor $\underline{Isom}^\otimes_K(K\otimes \omega_L, \rho)$ is
 trivial. \end{theorem}
 \begin{proof} $G:=\underline{Aut}^\otimes _k(\omega _L)$  coincides with the group of the
 $K$-linear differential automorphisms of $L$. Let $R$ denote the coordinate ring of the torsor
  $\underline{Isom}^\otimes_K(K\otimes \omega_L, \rho)$. Then $L$ is the field of fractions of $R$.\\

If $L$ is a real Picard--Vessiot field, then $R\subset L$ is a finitely generated real $K$-algebra. From the real
Nullstellensatz and the assumption that $K$ is real closed it follows that there exists a $K$-linear homomorphism
$\phi :R\rightarrow K$ with $\phi (1)=1$. The torsor $Spec(R)$ has a $K$-valued point and is therefore trivial.  \\

We observe that $L(i)$ is a Picard--Vessiot field for the differential module $K(i)\otimes M$ over $K(i)$. Further $G_{k(i)}$ is the group of the $K(i)$-linear differential automorphisms of $L(i)$ and is the `usual' differential Galois group of $K(i)\otimes M$ over $K(i)$. This group is connected since $K(i)$ is algebraically closed. \\

Suppose that the torsor $Spec(R)$ is trivial. Then $R\cong K\otimes _kk[G]\cong K[G]$. According to
 Lemma 2.2,  $K[G]$ is a real $K$-algebra and therefore its field of fractions  $L$ is a real field.  \end{proof}

\subsection{The final step}
By Lemma 2.1, we may suppose that $K$ is real closed. Let $L_1,L_2$ denote two real Picard-Vessiot fields for a differential module $M/K$.

 Write $\omega _j=\omega _{L_j}:<M>_\otimes \rightarrow vect(k)$ for the corresponding
fibre functors. Let $G=\underline{Aut}^\otimes_k(\omega _1)$. Then $\underline{Isom}^\otimes _k(\omega _1,\omega _2)$ is a $G$-torsor over $k$ corresponding to an element $\xi \in H^1(\{1,\sigma \},G(k(i)))$, where $\{1,\sigma\}$
is $Gal(k(i)/k)$, represented by a 1-cocycle $c$ with $c(1)=1,\ c(\sigma )\in G(k(i))$ and $c(\sigma)\cdot \ ^\sigma c(\sigma)=1$.

The $G_K$-torsor  $\underline{Isom}^\otimes _K(K\otimes \omega
_1,K\otimes \omega _2)$ corresponds to an
 element\linebreak
 $\eta \in H^1(\{1,\sigma\},G(K(i)))$. This element is the image of $\xi$ under the map, induced by the inclusion
 $G(k(i))\subset G(K(i))$, from  $H^1(\{1,\sigma \},G(k(i)))$ to $H^1(\{1,\sigma\},G(K(i)))$ (we note that
 $Gal(K(i)/K)=Gal(k(i)/k)$).   Since $L_j$ is real, the torsor
$\underline{Isom}^\otimes _K(K\otimes \omega _j,\rho )$ is trivial for $j=1,2$, by Theorem 1.5. Thus there exists
 isomorphisms $\alpha _j :K\otimes \omega _j\rightarrow \rho$ for $j=1,2$. The isomorphism
 $\alpha _2^{-1}\circ \alpha _1:K\otimes \omega _1\rightarrow K\otimes \omega _2$ implies that $\eta$ is trivial.
 In particular, there is an element $h\in G(K(i)))$ such that $c(\sigma )=h^{-1}\sigma (h)$.

  There exists a finitely generated $k$-algebra $B\subset K$ with $h\in G(B(i))$. Since $B$ is real and $k$ is real closed,
  there exists, by the real Nullstellensatz, a $k$-linear homomorphism $\phi :B\rightarrow k$ with $\phi (1)=1$.
  Applying $\phi$ to the identity $c(\sigma)=h^{-1}\sigma(h)$ one obtains $c(\sigma )=\phi (h)^{-1}\sigma (\phi (h))$.
  Thus $c$ is a trivial 1-cocycle and there is an isomorphism $\omega _1\rightarrow \omega _2$. Hence
  $L_1$ and $L_2$ are isomorphic as differential field extensions of $K$.  \\
  {\it Remark}. The natural map
  $H^1(\{1,\sigma\}, G(k(i)))\rightarrow H^1(\{1,\sigma\}, G(K(i)))$ is injective, by the above argument.   \hfill $\square$

\section{Comments and Examples}
        $\ $ \\
\indent  The proof of the unicity of a real Picard--Vessiot field uses almost exclusively properties of Tannakian
categories. This implies that the proof remains valid for other types of equations, such as:\\
(a). linear  partial  differential equations, like $\frac{\partial}{\partial z_j}Y=A_jY$ for $j=1,\dots ,n$,\\
(b). linear ordinary difference equations, like $Y(z+1)=AY(z)$, \\
(c). linear $q$-difference equations with $q\in \mathbb{R}^*$, like $Y(qz)=AY(z)$.\\

\noindent For  case  (a), the existence of a real
Picard-Vessiot field has been proved in \cite{CH}.
The proof of the uniqueness result (Theorem 1.1) for real
differential fields with real closed field of constants can
probably be rephrased  for the case of differential modules over a formally p-adic differential field with a p-adically closed field of constants of the
same rank.\\

\begin{observations} $\ $ \\ {\rm
Let $K$ be a real closed differential field with field of constants $k$,  $M/K$ a differential module and $\omega :<M>_\otimes \rightarrow vect(k)$ a fibre functor. Let $L$ be the Picard--Vessiot field
corresponding to $\omega$ and $G$ the group of the differential automorphisms of $L/K$. Let
$H$ be the differential Galois group of $K(i)\otimes M$ over $K(i)$. We recall that $G$ is a form of $H$ over the field $k(i)$. Using the
identification $k(i)\times _kG=H$, one obtains on $H$ and on $Aut(H)$ a structure of algebraic group over $k$. Let $\{1,\sigma\}$ be the Galois group of $k(i)/k$. Then $H^1(\{1,\sigma \}, Aut(H))$ has a natural bijection to the set of forms of $H$ over $k$. Although the action of
$\sigma$ on $Aut(H)$ depends on $G$, this set does not depend on the choice of $G$.

Let $\eta: <M>_\otimes \rightarrow vect(k)$ be another fibre functor. Then $\eta$ is mapped, according to Proposition 1.4, to an element in $\xi (\eta) \in H^1(\{1,\sigma \}, G(k(i)))$ (and this induces a bijection between $\eta$'s and elements in this cohomology set). A 1-cocycle $c$ for the group $\{1,\sigma\}$ has the form $c(1)=1,\ c(\sigma )=a$ and
$a$ should satisfy $a\cdot \sigma (a)=1$ (and is thus determined by $a$).

 A 1-cocycle for
$\xi (\eta)$ can be made as follows. The fibre functor $\eta$ corresponds
to  a Picard--Vessiot field $L_\eta$. Both $L(i)$ and $L_\eta (i)$ are
Picard--Vessiot fields for $K(i)\otimes M$ over $K(i)$. Thus there
exists a $K(i)$-linear differential isomorphism $\phi :L(i)\rightarrow
L_\eta (i)$. On the field $L(i)$ we write $\tau$ for the conjugation given
by $\tau(i)=-i$ and $\tau$ is the identity on $L$. The similar conjugation
on $L_\eta(i)$ is denoted by $\tau_\eta$.
 Now $\tau_\eta \circ \phi\circ \tau :L(i)\rightarrow L_\eta (i)$ is another
$K(i)$-linear differential isomorphism. A 1-cocycle $c$ for $\xi (\eta)$
is now  $c(\sigma )=\phi ^{-1} \circ \tau_\eta \circ \phi\circ \tau $.\\

Let $G_\eta$ denote the group of the $K$-linear differential automorphism of $L_\eta$. The group $G_\eta$ is a form of $G$ and
produces an element in $H^1(\{1,\sigma \},Aut(H))$ with $H=k(i)\times G$. We want to compute a 1-cocycle $C$ for this element.
Define the isomorphism $\psi: k(i)\times G\rightarrow k(i)\times G_\eta$
of algebraic groups over $k(i)$, by $\psi (g)=\phi \circ g\circ \phi ^{-1}$.
Define $\tau_G$, the `conjugation' on $k(i)\times G$, by the formula
$\tau_G(g)=\tau \circ g\circ \tau$ for the elements $g\in G(k(i))$.
Let $\tau_{G_\eta}$ be the similar conjugation on $k(i)\times G_\eta$.   Now $\tau_{G_{\eta}}\circ \psi \circ \tau_G:k(i)\times G\rightarrow
k(i)\times G_\eta$  is another isomorphism between the algebraic groups over $k(i)$. The 1-cocycle $C$ is given by
 $C(\sigma )=\psi ^{-1}\circ \tau_{G_{\eta}}\circ \psi \circ \tau_G$.
One observes that $C(\sigma )(g)=c(\sigma )gc(\sigma )^{-1}$.

The map, which associates to
$h\in G(k(i))$, the automorphism $g\mapsto hgh^{-1}$ of $G$, induces a map $H^1(\{1,\sigma\},G(k(i)))\rightarrow H^1(\{1,\sigma\},G/Z(G)(k(i)))
\rightarrow H^1(\{1,\sigma \}, Aut(H))$, denoted by
$\xi(\eta ) \mapsto \tilde{\xi}(\eta)$.
The forms corresponding to elements in the image of $H^1(\{1,\sigma\},G/Z(G)(k(i))) \rightarrow H^1(\{1,\sigma \}, Aut(H))$ are called `inner forms of $G$'.
By \S1,  $\eta$ induces a Picard--Vessiot field and a form $G(\eta )$ of $H$.  Above we have verified (see  \cite{B} for a similar computation) that $G(\eta )$ is the inner form of $G$ corresponding to the element $\tilde{\xi}(\eta)$.
For the  delicate theory of forms we refer to the informal manuscript
\cite{B} and the standard text \cite{Sp}. }\hfill $\square$
\end{observations}
\noindent {\bf Examples}.\\
 {\it We continue with the notation and assumptions  of Observations} (3.1). \\ (1). Let $M/K, \omega, L,G$ be such that $G={\rm SL}_{n,k}$. Since
$H^1(\{1,\sigma \}, {\rm SL}_{n}(k(i))$ is trivial, $L$ is the unique Picard--Vessiot field and is a real field (because a real Picard--Vessiot field exists).

The group ${\rm SL}_n$ has non trivial forms. For instance,
${\rm SU}(2)$ is an inner form of ${\rm SL}_{2,\mathbb{R}}$. There are
examples, according to Proposition 3.2 below, of differential modules $M/K$
having a real Picard--Vessiot field $L$ with group of differential automorphisms of $L/K$ equal to ${\rm SU}(2)$.

From \cite{Sp}, 12.3.7 and 12.3.9 one concludes that
$H^1(\{1,\sigma \},{\rm SU}(2)(\mathbb{C}))$ is trivial.  Again $L$ is the
only Picard--Vessiot field.\\
(2). If $G$ is the symplectic group ${\rm Sp}_{2n,k}$, then there are no forms and $H^1(\{1,\sigma \},G(k(i)))$ is trivial. Therefore there is only one Picard--Vessiot field $L$ and this is a real field.  \\
(3). Consider a $k$-form $G$ of ${\rm SO}(n)_{k}$ with odd $n\geq 3$. The center $Z$ of $G$ consists of the scalar matrices of order $n$, thus $Z$ is the group $\mu_{n,k}$ of the $n$th roots of unity. Since $n$ is odd, one has $Z(k)=\{1\}$. Further, again since $n$ is odd, the automorphisms of $H=G_{k(i)}$ are interior and $Aut(H)(k(i))=G/Z(k(i))$. We claim the following. \\

\noindent
{\it The natural map $H^1(\{1,\sigma \},G(k(i)) )\rightarrow
H^1(\{1,\sigma \},G/Z(k(i)) )$ is a bijection}.\\

\noindent
{\it Proof}. A 1-cocycle $c$ for $G/Z(k(i)) $ is given by $c(1)=1$ and
$c(\sigma )=a\in G/Z(k(i)) $ with $a\sigma (a)=1$. Choose an
$A\in G(k(i))$ which maps to $a$. Thus $A\sigma (A)\in Z(k(i))$ and
$A$ commutes with $\sigma A$. Further $\sigma (A\sigma (A))=
\sigma (A)A=A\sigma (A)$ and thus $A\sigma (A)\in \mu _n(k)=\{1\}$.
Therefore $C$ defined by $C(1)=1,\ C(\sigma )=A$ is a 1-cocycle for
$G(k(i))$ and maps to $c$. Hence the map is surjective.

Consider for $j=1,2$ the 1-cocycle $C_j$ for $G$ given by
$C_j(\sigma )=A_j$. Suppose that the images of $C_j$ as 1-cocycles for $G/Z(k(i))$ are equivalent. Then there exists  $B\in G(k(i))$ such that $B^{-1}A_1\sigma (B)=xA_2$ for some
element $x\in Z(k(i))$. We may replace $B$ by $yB$ with $y\in Z(k(i))$.
Then $x$ is changed into $xy^{-1}\sigma (y)$. And the latter is equal
to 1 for a suitable $y$. This proves the injectivity of the map.\hfill $\square$\\

We conclude from the above result  that there exists a (unique up to isomorphism) fibre functor $\eta :<M>_\otimes \rightarrow vect(k)$
(or, equivalently, a Picard--Vessiot field) for every form of $H=SO(n)_{k(i)}$ over $k$. Moreover, only one of these
fibre functors  corresponds to a {\it real} Picard--Vessiot field.

\bigskip

Let $\omega : <M>_\otimes \rightarrow vect(k)$ denote the fibre functor
corresponding to a {\it real} Picard--Vessiot field $L_\omega$ and
$G_\omega$ the group of the differential automorphisms of $L_\omega/K$.
{\it We want to  identify this form $G_\omega$ of
$H:={\rm SO}(n)_{k(i)}$}.\\

Since the differential Galois group of $K(i)\otimes M$ is ${\rm SO}(n)_{k(i)}$, there exists an element $F\in sym^2(K(i)\otimes M^*)$ with
$\partial F=0$. Further $F$ is unique up to multiplication by a scalar and
$F$ is a non degenerate bilinear symmetric form. The non trivial automorphism $\sigma$ of $K(i)/K$ and of $k(i)/k$ acts in an obvious way  on $K(i)\otimes M$ and  on constructions by linear algebra
of  $K(i)\otimes M$. Now $\sigma (F)$ has the same properties as $F$ and
thus $\sigma (F)=cF$ for some $c\in K(i)$. After changing $F$ into
$aF$ for a suitable $a\in K(i)$, we may suppose that $\sigma (F)=F$.
Then $F$ belongs to $sym^2(M^*)$ and is a non degenerate form of degree $n$ over the field $K$. Further $F$ is determined by its signature because $K$ is real closed. Moreover $KF$ is the unique
1-dimensional submodule of $sym^2(M^*)$. We claim the following:\\
  {\it
$G_\omega$ is the special orthogonal group over $k$ corresponding to
a form $f$ over $k$ which has the same signature as $F$}. \\

Let $V=\omega (M)$. The group
$G_\omega$ is the special
orthogonal group of some non degenerate bilinear symmetric form
$f\in sym^2(V^*)$. Since $L_\omega$ is real, there exists a  isomorphism
$m:K\otimes _k\omega \rightarrow \rho$ of functors. Applying $m$
to the modules $M$ and  $sym^2(M^*)$ one finds an isomorphism
$m_1:K\otimes _kV\rightarrow M$ of $K$-vector spaces
which induces an isomorphism of $K$-vector spaces
$m_2:K\otimes_k sym^2(V^*)\rightarrow sym^2(M^*)$.
The latter maps the subobject $K\otimes kf$ to $KF$ by the uniqueness of
$KF$.   One concludes that the forms $f$ and $F$ have the same signature.

\begin{proposition} Suppose that $K$ is real closed.
 Given is a connected semi-simple group $H$ over $k(i)$ and a form $G$ of $H$ over $k$. Then there exists a differential module $M$ over $K$ and a real Picard--Vessiot field for $M/K$ such that the group of the differential automorphisms of $L/K$ is $G$.\end{proposition}

 \begin{proof} Let $G$ be given as a subgroup of some ${\rm GL}_{n,k}$, defined by
 a radical ideal $I$. Then $k[G]=k[\{X_{k,l}\}_{k,l=1}^n,\frac{1}{\det }]/I$. The tangent space of
 $G$ at $1\in G$ can be identified with the $k$-linear derivations $D$ of this algebra, commuting with the action of $G$.  These derivations $D$ have the form $(DX_{k,l})=B\cdot (X_{k,l})$ for some
matrix $B\in Lie(G)(k)$ (where $Lie(G)\subset {\rm Matr}(n,k)$ is the Lie algebra of $G$).

The same holds for $K[G]=K\otimes _k k[\{X_{k,l}\}_{k,l=1}^n,\frac{1}{\det }]/I$. Any
$K$-linear derivation $D$ on the algebra, commuting with the action of $G$, has the form
$(DX_{k,l})=A\cdot (X_{k,l})$ with $A\in Lie(G)(K)$. We choose $A$ as general as possible.

The differential module $M/K$ is defined by the matrix equation
$y'=Ay$. It follows from \cite{PS}, Proposition 1.3.1 that the differential Galois group of $K(i)\otimes M$ is contained in $H=G_{k(i)}$.  Now one has to choose $A$ such that the differential Galois group (which is connected because $K(i)$ is algebraically closed) is not a proper subgroup of $H$. Since $H$ is semi-simple, there exists a
Chevalley module for $H$. Using this Chevalley module one can produce
a general choice of $A$ such that differential Galois
group of $y'=Ay$ over $K(i)$ is in fact $G_{k(i)}$ (compare \cite{PS}, \S 11.7 for the details which remain valid in the present situation).

The usual way to produce a Picard--Vessiot ring for the
equation  $y'=Ay$ is to consider the differential algebra
$R_0:=K[\{X_{k,l}\}_{k,l=1}^n,\frac{1}{\det }]$, with
differentiation defined by $(X'_{k,l})=A\cdot (X_{k,l})$, and to
produce a maximal differential ideal in  $R_0$.  Since
$A\in Lie(G)(K)\subset Lie(H)(K(i))$, the ideal $J\subset
R_0[i]$, generated by $I$ is a differential ideal.  It is in fact a maximal
differential ideal of $R_0[i]$, since the differential Galois group is precisely $H$.
Then $J\cap R_0=IR_0$ is a maximal differential ideal of $R_0$ and
$K[G]=R=R_0/IR_0$  is a Picard--Vessiot ring for $M$ over $K$.
The field of fractions $L$ of $R$ is real because the $G$-torsor
$Spec(K[G])$ is trivial.   \end{proof}

It seems that, imitating the proofs in \cite{MS}, one can show that
Proposition 3.2  remains valid under the weaker conditions: $K$ is a real differential field and a $C_1$-field and $H$ is connected.

\vspace{1cm}
\footnotesize

\begin{tabular}{lcl}
Teresa Crespo && Zbigniew Hajto  \\
Departament d'\`{A}lgebra i Geometria && Faculty of Mathematics and Computer Science \\  Universitat de Barcelona && Jagiellonian University \\ Gran Via de les Corts Catalanes 585 && ul. Prof. S. \L ojasiewicza 6 \\
08007 Barcelona, Spain && 30-348 Krak\'ow, Poland  \\
teresa.crespo@ub.edu && zbigniew.hajto@uj.edu.pl
\end{tabular}

\vspace{1cm}
\begin{tabular}{l}
Marius van der Put \\
Department of Mathematics \\
University of Groningen\\
P.O. Box 800 \\
9700 AV Groningen, The Netherlands \\
M.van.der.Put@math.rug.nl
\end{tabular}
\end{document}